\documentclass[11pt]{amsart}

\usepackage{amsmath,amsfonts,amsthm,amsopn,amssymb}
\usepackage{graphicx}
\usepackage{epsfig,verbatim}
\usepackage{mathrsfs}
\usepackage{hyperref}

\newtheoremstyle{mythm}{6pt}{6pt}{\itshape}{}{\bfseries}{.}{ }{} 
\newtheoremstyle{mydef}{6pt}{6pt}{}{}{\bfseries}{.}{ }{}         
\newtheoremstyle{myrem}{6pt}{6pt}{}{}{\scshape}{.}{ }{}          

\theoremstyle{mydef}
\newtheorem{definition}{Definition}[section]

\theoremstyle{mythm}
\newtheorem{theorem}[definition]{Theorem}
\newtheorem{proposition}[definition]{Proposition}
\newtheorem{lemma}[definition]{Lemma}
\newtheorem{corollary}[definition]{Corollary}

\theoremstyle{myrem}
\newtheorem{remark}[definition]{Remark}
\newtheorem{example}[definition]{Example}

\renewenvironment{proof}{\par\medbreak\noindent\textit{Proof}:\hskip.5em\ignorespaces}{\hfill\qedsymbol\medbreak} 
\renewcommand{\qedsymbol}{\hfill\rule{2mm}{2mm}}                                                                  

\newcommand{\R}{\mathbb{R}}

\def\cV{\mathcal{V}}

\newcommand{\Rt}{\mathbb{R}^{2}}

\newcommand{\So}{\boldsymbol S_0}
\newcommand{\mfrac}[2]{\text{\footnotesize{$\dfrac{#1}{#2}$}}}
\newcommand{\lfrac}[2]{\text{\small{$\dfrac{#1}{#2}$}}}

\allowdisplaybreaks

\begin{document}
\title[Samp.~TF-loc.~functions and const.~loc.~TF frames]{Sampling time-frequency localized functions and constructing localized time-frequency frames}

\author[M.~D\"{o}rfler]{Monika D\"{o}rfler}
\address{Numerical Harmonic Analysis Group, Faculty of Mathematics, University of Vienna, Oskar-Morgenstern-Platz 1, 1090 Vienna, Austria}
\email[Monika D\"{o}rfler]{monika.doerfler@univie.ac.at}

\author[G.~Velasco]{Gino Angelo Velasco}
\address{Institute of Mathematics, University of the Philippines, Diliman, Quezon City 1101, Philippines}
\email[Gino Angelo Velasco]{gamvelasco@math.upd.edu.ph}

\thanks{The first author was supported by the Vienna Science and Technology Fund (WWTF) through project MA14-018. The second author was partially supported by the FWF Einzelprojekte P 27773.}


\begin{abstract}
We study functions whose time-frequency content are concentrated in a compact region in phase space using time-frequency localization operators as a main tool. We obtain approximation inequalities for such functions using a finite linear combination of eigenfunctions of these operators, as well as a local Gabor system covering the region of interest. These would allow the construction of modified time-frequency dictionaries concentrated in the region.
\end{abstract}

\maketitle

\section{Introduction}
When processing audio signals such as music or speech, one sometimes strives for a meticulous separation of signal components in time and frequency. 
However, it is known that no nonzero function can be compactly supported simultaneously in time and frequency, so that good concentration in  one domain usually has to be paid for with increased leakage in the other domain. Due to this trade-off, window design has been an important issue in signal processing. 
As opposed to traditional  approaches, we investigate the optimization of concentration simultaneously in time and frequency. More precisely, we investigate  functions that exhibit good concentration  in a compact region in the time-frequency plane. Our method is related to the approach introduced by    Landau, Slepian, and Pollak, who  considered operators composed of consecutive time- and bandlimiting steps, cf. \cite{posl61,lapo61,lapo62}. The resulting operators  yield the well-known prolate spheroidal functions as eigenfunctions. These functions satisfy some optimality in concentration in a rectangular region in the time-frequency domain. 

In \cite{da88}, Daubechies introduced time-frequency localization operators obtained by restricting the synthesis from time-frequency coefficients  to a desired region of interest. Here, we make use of time-frequency localization operators to describe a function's local time-frequency content in regions  more general than the rectangles considered in \cite{posl61,lapo61,lapo62}. As in the case of the prolate spheroidal wave functions, the eigenfunctions of the time-frequency localization operators are maximally time-frequency-concentrated in  the region of interest. We will use these eigenfunctions to characterize a function's time-frequency localization.

Using Gabor frames, we show, to which extent time-frequency localized functions can be approximated using only a finite number of Gabor coefficients, namely, those which are inside some larger  cover of the given region. This is influenced by the approximation result formulated by Daubechies in a seminal paper \cite{da90}. Similar estimates were also established in \cite{elma12}.

Using the obtained results, we construct global time-frequency frames consisting of atoms which are optimally concentrated in small regions corresponding to a prescribed lattice.   
Each locally concentrated  system is constructed by projecting the local Gabor atoms onto the subspace spanned by a finite number of eigenfunctions. Then,  by considering a family of such locally concentrated time-frequency dictionaries, we  obtain an adaptive frame for all $L^2(\mathbb{R})$.  The resulting frames are useful for processing tasks in which, as explained above, certain signal components have to be processed separately with minimum distortion of close-by signals parts. We will give an application example in the experiments Section \ref{sec:NumEx}. Note that a similar construction involving prolate spheroidal functions has recently been studied in \cite{hola15}.

In the next section, we recall the necessary tools from time-frequency analysis, namely, the short-time Fourier transform and Gabor frames. In Section \ref{sec:TFConcent}, we review some properties of time-frequency localization operators, and we prove characterization and approximation results concerning time-frequency localized functions and eigenfunctions of time-frequency localization operators. Section \ref{sec:LocGabApp} deals with approximation using a local Gabor system and the construction of the new time-frequency dictionaries. Finally, in Section \ref{sec:NumEx}, we present  numerical experiments concerning our results and conclude with some perspectives in Section \ref{sec:Conc}.
\section{Preliminaries}
In this section we recall some definitions and properties about the short-time Fourier transform and Gabor frames. For a thorough introduction to the field of time-frequency analysis, we refer the reader to \cite{gr01}. 

Fix a window function $\varphi\in L^2(\R)$ with norm $\|\varphi\|_2 = 1$. The \emph{short-time Fourier transform} (STFT) of $f\in L^2(\mathbb{R})$ with respect $\varphi$ is given by 
\begin{equation}\label{eq:STFT}
\cV_{\varphi}f(z) = \cV_{\varphi}f(x,\omega) = \int_{\R}f(t)\,\overline{\varphi(t-x)}\,e^{-2\pi i\omega\cdot t}dt=\langle f,\pi(z)\varphi\rangle,
\end{equation}
 where $z=(x,\omega)\in\Rt$ and $\pi(z)$ is the time-frequency shift operator given by $\pi(z)f = f(t-x)\,e^{2\pi i\omega\cdot t}$. The STFT is an isometry from $L^2(\R)$ to $L^2(\Rt)$, i.e. $\|\cV_{\varphi}f\|_2 = \|\varphi\|_2\|f\|_2$, and inversion is realized using the formula
\begin{equation}\label{eq:InvForm}
 f = \cV_{\varphi}^{\ast}\,\cV_{\varphi}f = \iint_{\Rt} \cV_{\varphi}f(z)\pi(z)\varphi\,dz.
\end{equation}

The membership of the STFT in $L^p(\Rt)$ provides a definition of a class of function spaces called \emph{modulation spaces}. In particular, for a fixed non-zero window function $\varphi\in\mathcal{S}(\R)$, the modulation space $M^p(\R)$ is defined as the space of all tempered distributions $f\in\mathcal{S}'(\R)$ such that $\cV_{\varphi}f\in L^p(\Rt)$. It is a Banach space equipped with the norm $\|f\|_{M^p(\R)} := \|\cV_{\varphi}f\|_{L^p(\Rt)}$, where a different window function $\varphi$ would yield an equivalent norm. Note that for $\varphi\in L^2(\R)$ with $\|\varphi\|_2 = 1$, the isometry of the STFT implies that $M^2(\R) = L^2(\R)$. The space $\So(\R) = M^1(\R)$ is also known as Feichtinger's algebra. It is the smallest Banach space isometrically invariant under time-frequency shifts and the Fourier transform, cf.~\cite{fezi98}.

Discretization of the time-frequency representation via the STFT leads to the theory of Gabor frames. We consider a \emph{Gabor system} $\mathcal{G}(g,\Lambda)$ with window function $g\in L^2(\R)$ and a countable set of points $\Lambda$ in $\Rt$,  consisting of time-frequency shifted copies of a function $g$, i.e. $\mathcal{G}(g,\Lambda):=\{g_{\lambda}:=\pi(\lambda)g\}_{\lambda\in\Lambda}$. We say that $\mathcal{G}(g,\Lambda)$ is a \emph{frame} for $f\in L^2(\R)$ if there exist constants $A,\,B>0$ such that for all $f\in L^2(\R)$
\begin{equation}\label{eq:FrIneq}
  A\|f\|_2^2\leq \sum\limits_{\lambda\in\Lambda}|\langle f,g_{\lambda}\rangle|^2 \leq B\|f\|_2^2.
\end{equation}
If $A=B$, then we say that $\mathcal{G}(g,\Lambda)$ is a \emph{tight frame}.

Associated with the frame $\mathcal{G}(g,\Lambda)$ is the \emph{frame operator} $S$ given by \begin{equation*}
 Sf = \sum_{\lambda\in\Lambda} \langle f,g_{\lambda}\rangle g_{\lambda}.
\end{equation*}
The frame conditions \eqref{eq:FrIneq} are equivalent to the invertibility of $S$, and reconstruction from the coefficients $\{\langle f,g_{\lambda}\rangle\}_{\lambda\in\Lambda}$ is possible because of the existence of a dual frame $\{\tilde{g_{\lambda}}\}_{\lambda\in\Lambda}$, the canonical one being $\{S^{-1}g_{\lambda}\}_{\lambda\in\Lambda}$, having frame bounds $B^{-1}$ and $A^{-1}$. Every $f\in L^2(\R)$ will then have the expansion 
  \begin{equation}\label{eq:FrExp}
   f = \sum_{\lambda\in\Lambda}\langle f,g_{\lambda}\rangle \tilde{g_{\lambda}} = \sum_{\lambda\in\Lambda}\langle f,\tilde{g_{\lambda}}\rangle g_{\lambda}, 
  \end{equation}
where both series converge unconditionally in $L^2(\R)$. If $\mathcal{G}(g,\Lambda)$ is a tight frame, then $Sf = Af$ so $f = \frac{1}{A}Sf$.

Moreover, if $\mathcal{G}(g,\Lambda)$ is a frame, then the \emph{analysis operator} $U$ given by $Uf = \{\langle f,g_{\lambda}\rangle\}_{\lambda\in\Lambda}$ and its adjoint $U^{\ast}$, called the \emph{synthesis operator}, given by $U^{\ast}c = \sum_{\lambda\in\Lambda}c_{\lambda}g_{\lambda},\,c = \{c_{\lambda}\}_{\lambda\in\Lambda}$, are bounded from $L^2(\R)$ into $\ell^2(\Lambda)$ and $\ell^2(\Lambda)$ into $L^2(\R)$, respectively, with operator norms $\|U\|_{\text{Op}} = \|U^{\ast}\|_{\text{Op}} \leq \sqrt{B}$. For the dual frame $\{\tilde{g_{\lambda}}\}_{\lambda\in\Lambda}$, the associated analysis and synthesis operators are also bounded operators with operator norms not exceeding $1/\sqrt{A}$.

\section{Time-frequency concentration via the STFT}\label{sec:TFConcent}
Time-frequency localization operators as introduced by Daubechies in \cite{da88} are built by restricting the integral in the inversion formula \eqref{eq:InvForm} to a subset of $\Rt$. Its properties, connections with other mathematical topics, and applications have been topics in various works, e.g. \cite{rato94,feno01,defeno02,wo02,cogr03,abdo12,grto13,doro13,doro14}.

Let $\Omega$ be a compact set in $\Rt$ and $\varphi$ a window function in $L^2(\R)$, with $\|\varphi\|_2 = 1$. The \emph{time-frequency localization operator} $H$ is defined by
\begin{equation}\label{eq:TFLocOpDef}
  Hf = \iint_{\Omega} \cV_{\varphi}f(z)\pi(z)\varphi\,dz = \cV_{\varphi}^{\ast}\,\chi_{\Omega}\,\cV_{\varphi}f.
\end{equation}
Note that while we denote a time-frequency localization only by $H$, we emphasize that it is dependent on the window $\varphi$ and the region $\Omega$. 

The above integral can be interpreted as the portion of the function $f$ that is essentially contained in $\Omega$. Moreover, the following inner product involving $H$ measures the function's energy inside $\Omega$:
\begin{equation*}\label{eq:TFLocMeas}
  \langle Hf,f\rangle = \iint_{\Omega}\cV_{\varphi}f(z)\langle \pi(z)\varphi,f\rangle dz = \iint_{\Omega}|\cV_{\varphi}f(z)|^2 dz.
\end{equation*}
We will say that a function $f\in L^2(\R)$ is \emph{$(\varepsilon,\varphi)$-concentrated} inside $\Omega$ if  $\langle Hf,f\rangle\geq (1-\varepsilon)\|f\|_2^2$ or equivalently $\langle (I-H)f,f\rangle\leq \varepsilon\|f\|_2^2$, where $I$ is the identity operator. 

The time-frequency localization operator $H$ is a compact and self-adjoint operator so we can consider the spectral decomposition 
\begin{equation}
  Hf = \sum_{k = 1}^{\infty}\alpha_k\langle f,\psi_k\rangle \psi_k,
\end{equation}
where $\{\alpha_k\}_{k=1}^{\infty}$ are the positive eigenvalues arranged in a non-increasing order and $\{\psi_k\}_{k = 1}^{\infty}$ are the corresponding eigenfunctions.
By the min-max theorem for compact, self-adjoint operators, the first eigenfunction, has optimal time-frequency concentration inside $\Omega$ in the sense of \eqref{eq:TFLocMeas}, i.e.
\begin{equation*}
 \iint_{\Omega}|\cV_{\varphi}\psi_1(z)|^2 dz = \max_{\|f\|_2 = 1} \iint_{\Omega}|\cV_{\varphi}f(z)|^2 dz.
\end{equation*}
In general, the first $N$ eigenfunctions $\psi_1,\,\psi_2,\ldots,\,\psi_N$ form an orthonormal set in $L^2(\R)$ having optimal cumulative time-frequency concentration inside $\Omega$:
\begin{equation*}
 \sum_{k = 1}^N\iint_{\Omega}|\cV_{\varphi}\psi_k(z)|^2 dz = \max_{\{\phi_k\}_{k=1}^N \text{orthonormal}}\sum_{k = 1}^N\iint_{\Omega}|\cV_{\varphi}\phi_k(z)|^2 dz.
\end{equation*}
If we let $V_N$ be the span of the first $N$ eigenfunctions and if $f\in V_N$ so $f = \sum_{k = 1}^N\langle f,\psi_k\rangle \psi_k$, then 
\begin{equation}\label{eq:HfInnerProd}
  \langle Hf,f\rangle = \sum_{k=1}^N\alpha_k|\langle f,\psi_k\rangle|^2\geq \alpha_N\sum_{k = 1}^N|\langle f,\psi_k\rangle|^2 = \alpha_N\|f\|_2^2.
\end{equation}
This implies that a function $f$ in $V_N$ is at least $(1-\alpha_N,\varphi)$-concentrated on $\Omega$, and any other orthonormal, $N$-dimensional subspace cannot be better concentrated than $(1-\alpha_N,\varphi)$-concentrated on $\Omega$.

By contrast, functions which are $(1-\alpha_N,\varphi)$-concentrated on $\Omega$ need not lie in $V_N$. The following proposition characterizes a function that is $(\varepsilon,\varphi)$-concentrated on $\Omega$.

\begin{proposition}Let $\varphi$,   $\Omega$ and $\varepsilon$ be given and let $N_0$ be the integer such that $\alpha_{N_0}\geq 1-\varepsilon$ and $\alpha_{N_0+1}< 1-\varepsilon$. Furthermore, let  $f_{ker}$ denote  the orthogonal projection of $f$ onto the kernel $\ker(H)$ of $H$.
    A function $f$ in $L^2(\R)$ is $(\varepsilon,\varphi)$-concentrated on $\Omega$ if and only if
\[    \sum\limits_{k = 1}^{N_0}(\alpha_k+\varepsilon-1)|\langle f,\psi_k\rangle|^2\geq \sum\limits_{k = N_0+1}^{\infty}(1-\varepsilon-\alpha_k)|\langle f,\psi_k\rangle|^2 +(1-\varepsilon)\|f_{ker}\|_2^2.
\]
\end{proposition} 
\begin{proof}
The eigenfunctions $\{\psi_k\}_k$ form an orthonormal subset in $L^2(\R)$, possibly incomplete if $\ker(H)\neq \{0\}$; hence, we can write 
$f = \sum_{j=1}^{\infty}\langle f,\psi_j\rangle\psi_j +f_{ker}$ and, as in \eqref{eq:HfInnerProd}, $\langle Hf,f\rangle = \sum_{k=1}^{\infty}\alpha_k|\langle f,\psi_k\rangle|^2$.
  So the function $f$ is $(\varepsilon,\varphi)$-concentrated on $\Omega$ if and only if 
  \begin{equation*}
    \sum_{k=1}^{\infty}\alpha_k|\langle f,\psi_k\rangle|^2\geq (1-\varepsilon)\left(\sum_{k=1}^{\infty}|\langle f,\psi_k\rangle|^2 + \|f_{ker}\|_2^2\right),
  \end{equation*}
  and the conclusion follows.
\end{proof}
  
\begin{remark}
 A function $f$ in $L^2(\Rt)$ is $(1-\alpha_N,\varphi)$-concentrated on $\Omega$ if and only if $\sum\limits_{k = 1}^{N-1}(\alpha_k-\alpha_N)|\langle f,\psi_k\rangle|^2\geq \sum\limits_{k = N+1}^{\infty}(\alpha_N-\alpha_k)|\langle f,\psi_k\rangle|^2 + \alpha_N\|f_{ker}\|_2^2.$
\end{remark}

In \cite[Theorem 3.1]{da90}, Daubechies bounded the error of a function's local approximation using a finite number  of Gabor atoms by means of an estimate based on  the function's and its Fourier transform's {\em projection} onto bounded intervals. In order to achieve a similar bound, but  for more general regions, in Proposition~\ref{Pro:TFlocByFrames} in the next section, we consider the projection of a function $f$ onto the best-concentrated eigenfunctions of a localization operator and derive the following estimate. We note that approximations of bandlimited functions via projections onto eigenspaces of approximately time- and bandlimited functions were presented in \cite{shwa03,hoizla10}.

\begin{proposition}\label{Prop:eigFuncApp}
  Let $f$ be $(\varepsilon,\varphi)$-concentrated on $\Omega\subset\Rt$. For fixed $c>1$, let $\psi_k,\,k = 1,\ldots,N$, be all eigenfunctions of $H$ corresponding to eigenvalues $\alpha_k>\frac{c-1}{c}$. Then 
  \begin{equation*}
    \bigg\|f-\sum\limits_{k=1}^N\langle f,\psi_k\rangle\psi_k\bigg\|_2^2<c\varepsilon\|f\|_2^2.
  \end{equation*}  
\end{proposition}
\begin{proof}
  Without loss of generality, we assume that $\|f\|_2 =1$. We have, by assumption:
  \begin{equation*}
   \langle Hf,f\rangle = \sum_{k = 1}^{\infty}\alpha_k|\langle f,\psi_k\rangle|^2 = \iint_{\Omega}|\cV_{\varphi}f(z)|^2dz\geq (1-\varepsilon)\|f\|_2^2
  \end{equation*}
We argue by contradiction; to this end,  assume that $\sum_{k = 1}^N|\langle f,\psi_k \rangle|^2=K < 1-c\varepsilon$.
Furthermore
 \begin{equation*}
 \| f\|_2^2 = 1 =    \sum_{k = 1}^{\infty}|\langle f,\psi_k \rangle|^2+ \| f_{ker}\|_2^2 \, , 
  \end{equation*}
  hence
   \begin{equation*}
    \sum_{k = N+1}^{\infty} |\langle f,\psi_k \rangle|^2  =1-K-\| f_{ker}\|_2^2.
  \end{equation*}
   We then have
  \begin{equation*}
    \sum_{k = N+1}^{\infty}\alpha_k|\langle f,\psi_k \rangle|^2<\frac{c-1}{c} \cdot(1-K-\| f_{ker}\|_2^2)
  \end{equation*}
  such that 
  \begin{align*}
    \sum_{k = 1}^{\infty}\alpha_k|\langle f,\psi_k \rangle|^2&<K+ \frac{c-1}{c} \cdot(1-K-\| f_{ker}\|_2^2)\\
    &=\frac{c-1+K}{c}  - \frac{c-1}{c} \| f_{ker}\|_2^2\\
     &<1+\frac{1-c\varepsilon-1}{c}  - \frac{c-1}{c} \| f_{ker}\|_2^2<1-\varepsilon,
  \end{align*}
  which is a contradiction. Hence, $\sum_{k = 1}^{N}|\langle f,\psi_k \rangle|^2$ must be greater than or equal to $1-c\varepsilon$.\\ \noindent\phantom{-}
\end{proof}

We note that while $Hf$ is interpreted as the part of $f$ in $\Omega$, the uncertainty principle prohibits its STFT to have nonzero values only in $\Omega$, cf.~\cite{wi00}, and there will always be points $z\in \Rt\setminus\Omega$ at which $|\cV_{\varphi}Hf(z)|\neq 0$. It can be shown, however, that $|\cV_{\varphi}Hf(z)|$ decays fast with respect to the distance of $z$ from $\Omega$. Daubechies proved this result in \cite{da88} for the case where the window function is the Gaussian $\varphi_0(t) = e^{-\pi t^2}$, showing that the pointwise magnitude of the STFT decays exponentially.

\begin{lemma}[Daubechies, \cite{da88}]\label{lemma:0}
  Let $H$ be a time-frequency localization operator over the region $\Omega$ with the Gaussian as its window function, i.e.~$\varphi = \varphi_0$.  For any $\delta$ between $0$ and $1$, and $f\in L^2(\R)$, one has
  \begin{equation*}
   |\langle Hf,\pi(z)\varphi_0\rangle|\leq \tfrac{1}{\sqrt{2}}\delta^{-\tfrac{1}{2}}\|f\|_2\exp[-\tfrac{\pi}{2}(1-\delta)\operatorname{dist}(z,\Omega)^2].
  \end{equation*}
\end{lemma}

A similar result involving windows with milder decay conditions is the following.

\begin{lemma}\label{lemma:1}
  Let $\varphi,\,g\in L^2(\R)$ such that $\|\varphi\|_2 = 1$ and $|\cV_{\varphi}g(z)|\leq C(1+|z|^{2s})^{-1}$, for some $C>0$ and $s>1$, for all $z\in\Rt$. For any $\delta$ between $0$ and $1$, one has
  \begin{equation*}
   |\cV_{\varphi}Hf(z)|=|\langle Hf,\pi(z)g\rangle|\leq C_s\delta^{-\tfrac{1}{2s}}\|f\|_2(1+(1-\delta){\operatorname{dist}}(z,\Omega)^s)^{-1},
  \end{equation*}
  where $C_s = \frac{C\sqrt{2}\,\pi}{\sqrt{s\sin(\pi/s)}}$.
\end{lemma}

\begin{remark}
  An example of the inequality $|\cV_{\varphi}g(z)|\leq C(1+|z|^{2s})^{-1}$ being satisfied for all $z\in\Rt$ is when $\varphi$ and $g$ are in the Schwartz space $\mathcal{S}(\R)$. Moreover, in that case, for every $s>0$, there is a $C$ for which the inequality is satisfied.
\end{remark}

\begin{proof}
  If $z,\,z'\in\mathbb{R}^2$, then $|\langle \pi(z')\varphi,\,\pi(z)g\rangle|=|\langle \varphi,\pi(z-z')g\rangle|\leq C_s'(1+|z-z'|^{2s})^{-1}$.

  For $0<\delta<1$,
  \begin{align*}
   |\langle Hf,\pi(z)g\rangle| &\leq \iint_{\Omega}|\langle f,\pi(z')\varphi \rangle|\,|\langle \pi(z')\varphi,\pi(z)g\rangle|\,dz'\\
    &\leq C \iint_{\Omega}|\langle f,\pi(z')\varphi \rangle|\dfrac{1}{1+|z-z'|^{2s}}\,dz'\\
    &\leq C\iint_{\Omega}|\langle f,\pi(z')\varphi \rangle|\dfrac{1}{\sqrt{1+\delta|z-z'|^{2s}}}\dfrac{1}{\sqrt{1+(1-\delta)|z-z'|^{2s}}}\,dz'\\ 
    &\leq C\sqrt{2}\dfrac{1}{1+(1-\delta)\inf\limits_{z'\in\Omega}|z-z'|^s}\left(\iint_{\Rt} \dfrac{1}{1+\delta|z-z'|^{2s}}\,dz'\right)^{\tfrac{1}{2}}\cdot\\
    &\phantom{\leq C\sqrt{2}} \left(\iint_{\Rt}|\langle f,\pi(z')\varphi\rangle|^2\,dz'\right)^{\tfrac{1}{2}}\\ 
    &=\frac{C\sqrt{2}\,\pi}{\sqrt{s\sin(\pi/s)}}\delta^{-\tfrac{1}{2s}}(1+(1-\delta)\inf\limits_{z'\in\Omega}|z-z'|^s)^{-1}\|\varphi\|_2\|f\|_2,
  \end{align*}
  and the conclusion follows.
\end{proof}

\section{Local Gabor approximation and new TF dictionaries}\label{sec:LocGabApp}
We shall make use of the results concerning time-frequency localization to obtain an approximation of a function using a finite Gabor expansion. We expect that if the function is well localized on a region $\Omega$ in the time-frequency plane, then $f$ can be approximated with good accuracy using only the Gabor coefficients on a larger region covering $\Omega$.  From the local Gabor systems, we will obtain frames for the eigenspace, the collection of which forms a frame for $L^2(\R)$.

\subsection{Time-frequency localization and local Gabor approximation}\label{SubSec:LocGab}
In this section, we consider a given region $\Omega$ in $\mathbb{R}^2$ and
let $V_N$ be the $N$-dimensional subspace spanned by the first $N$ eigenfunctions of the corresponding localization operator $H$. Here, the eigenvalues are assumed to be  arranged in descending order. Furthermore, we let $g$ be a window function in $L^2(\R)$ such that $\|g\|_2=1$ and $|\cV_{\varphi}g(z)|\leq C(1+|z|^{2s})^{-1}$ for some $C>0$ and $s>1$. We then consider the Gabor system $\mathcal{G}(g,\Lambda)$, assume that it forms  a frame with lower and upper frame bounds $A$ and $B$, respectively, and let $\{\tilde{g_{\lambda}}\,:\,\lambda\in\Lambda\}$ be its dual frame. 

We first have the following approximation of a function in $V_N$ local Gabor atoms. 

\begin{proposition}\label{prop:LocGabAppTFLocSubsp}
  For any $\varepsilon>0$, there exists an $\Omega^*\supset\Omega$ such that
  \begin{equation}\label{eq:2est}
    \left\|f-\sum\limits_{\lambda\in\Lambda\cap\Omega^*}\langle f,g_{\lambda}\rangle\tilde{g_{\lambda}}\right\|_2\leq \varepsilon\|f\|_2,\quad\text{for all }f\in V_N.
  \end{equation}
\end{proposition}

\begin{proof}
  Let $\varepsilon>0$ and $f\in V_N$. We first observe that for any $\Omega^{\ast}\supset \Omega$,
    \begin{align*}
      \bigg\|f-\sum\limits_{\lambda\in\Lambda\cap\Omega^*} \langle f,g_{\lambda}\rangle \tilde{g_{\lambda}}\bigg\|_2^2 &= \bigg\|\sum\limits_{k=1}^N\langle f,\psi_k\rangle\bigg(\sum\limits_{\lambda\notin\Lambda\cap\Omega^*}\langle \psi_k,g_{\lambda}\rangle \tilde{g_{\lambda}}\bigg)\bigg\|_2^2\\
      &\leq \bigg(\sum\limits_{k=1}^N|\langle f,\psi_k\rangle|\bigg\|\sum\limits_{\lambda\notin\Lambda\cap\Omega^*}\langle \psi_k,g_{\lambda}\rangle \tilde{g_{\lambda}}\bigg\|_2\bigg)^2\\
      &\leq \sum\limits_{k=1}^N|\langle f,\psi_k\rangle|^2\sum\limits_{k=1}^N\bigg\|\sum\limits_{\lambda\notin\Lambda\cap\Omega^*}\langle \psi_k,g_{\lambda}\rangle \tilde{g_{\lambda}}\bigg\|_2^2\\
      &\leq A^{-1}\|f\|_2^2\sum\limits_{k=1}^N\sum\limits_{\lambda\notin\Lambda\cap\Omega^*}|\langle\psi_k,g_{\lambda}\rangle|^2.
    \end{align*}
  We consider $|\langle\psi_k,g_{\lambda}\rangle|$ and note that $|\langle\psi_k,g_{\lambda}\rangle|= \tfrac{1}{\alpha_k}|\langle H\psi_k,g_{\lambda}\rangle|$. Since $g$ satisfies $|\cV_{\varphi}g(z)|\leq C(1+|z|^{2s})^{-1}$, it follows from Lemma \ref{lemma:1} that $|\langle\psi_k,g_{\lambda}\rangle|\leq \frac{1}{\alpha_k}C_s2^{\frac{1}{2s}}(1+\tfrac{1}{2}\operatorname{dist}(\lambda,\Omega)^s)^{-1}$, where $\delta$ is taken to be $\frac{1}{2}$, which gives us 
    \begin{align*}
     \bigg\|f-\sum\limits_{\lambda\in\Lambda\cap\Omega^*} \langle f,g_{\lambda}\rangle \tilde{g_{\lambda}}\bigg\|_2^2&\leq A^{-1}\|f\|_2^2C_s^2 2^{\frac{1}{s}}\bigg(\sum_{k=1}^N\lfrac{1}{\alpha_k^2}\bigg)\sum\limits_{\lambda\notin\Lambda\cap\Omega^*}(1+\tfrac{1}{2}\operatorname{dist}(\lambda,\Omega)^s)^{-2}.
    \end{align*}
  The right-hand side of the above inequality approaches $0$ as $\Omega^{\ast}$ gets larger. In particular, given $\varepsilon>0$, one can choose $\Omega^*$ so that the sum 
  \begin{equation*}
      \sum\limits_{\lambda\notin\Lambda\cap\Omega^*}(1+\tfrac{1}{2}\operatorname{dist}(\lambda,\Omega)^s)^{-2}< \varepsilon^2/\bigg(A^{-1}C_s^2 2^{\frac{1}{s}}\sum\limits_{k=1}^N\lfrac{1}{\alpha_k^2}\bigg)
  \end{equation*}
   which gives the conclusion of the proposition.
\end{proof}

We show an example for the case where the window function is the Gaussian $\varphi_0$ and the region $\Omega$ is the disk $B(O,R)$ with center at the origin $O$ and with radius $R$. We will make use of the decay of the STFT of $H$ in Lemma \ref{lemma:0}. First, we prove the following lemma that gives an estimate on the decay of the tail of the sum of samples of the two-dimensional Gaussian outside the disk $B(O,R^{\ast})$. Let $Q(j) = [j_1-\frac{1}{2},j_1+\frac{1}{2})\times[j_2-\frac{1}{2},j_2+\frac{1}{2}),\,j = (j_1,j_2)\in\mathbb{Z}^2$.

\begin{lemma}\label{lem:EstDecaySumExp}
  Let $\Lambda$ be a relatively separated set of points in $\mathbb{R}^2$ with $\sup_{z\in\mathbb{R}^2}\#(\Lambda\cap Q(z))=:N_{\Lambda}<\infty$. Fix $R>0$. If $R^{\ast}>R$, then
  \begin{equation}\label{eq:EstDecaySumExp}
    \sum_{\lambda\in\Lambda,\,|\lambda|>R^{\ast}}\exp(-\tfrac{\pi}{2}(|\lambda|-R)^2)\leq C_{\Lambda} \exp\big(-\tfrac{\pi}{4}\big(\tfrac{(R^{\ast})^2}{4}-R^2\big)\big),
  \end{equation}
  where $C_{\Lambda} = 8\exp(\frac{5\pi}{4})N_{\Lambda}$.
\end{lemma}

\begin{proof}
  Let $R^{\ast}>R$ and define the sets 
  \begin{align*}
    \mathcal{J}_{R^{\ast}} &= \{j\in\mathbb{Z}^2\,:\,Q(j)\cap (\Rt\setminus B(O,R^{\ast}))\neq\varnothing\}\,\,\text{ and}\\
    \Lambda_{R^{\ast},j} &= \{\lambda\in\Lambda\,:\, |\lambda|>R^{\ast},\,\lambda\in Q(j)\}.
  \end{align*}
  We are then able to rewrite the left-hand side of \eqref{eq:EstDecaySumExp} as 
  \begin{equation}\label{eq:SumExpSamp}
    \sum_{\lambda\in\Lambda,\,|\lambda|>R^{\ast}}\exp(-\tfrac{\pi}{2}(|\lambda|-R)^2) = \sum_{j\in \mathcal{J}_{R^{\ast}}}\sum_{\lambda\in\Lambda_{R^{\ast},j}}\exp(-\tfrac{\pi}{2}(|\lambda|-R)^2).
  \end{equation}
  If $\lambda,z\in Q(j)$, then $|\lambda|\geq |z|-\sqrt{2}$. And since $-(|z|-R-\sqrt{2})^2\leq -\big(\tfrac{(|z|-R)^2}{2}-2\big)$, we have
  \begin{equation*}
    e^{-\frac{\pi}{2}(|\lambda|-R)^2}\leq e^{-\frac{\pi}{2}(|z|-R-\sqrt{2})^2}\leq e^{\pi}\exp\big(-\tfrac{\pi}{4}\big(|z|-R\big)^2\big).
  \end{equation*}
  Using the inequalities $-(|z|-R)^2\leq -\frac{|z|^2}{2}+R^2$\, and \,$-(R^{\ast}-\sqrt{2})^2\leq -\frac{(R^{\ast})^2}{2}+2$, we estimate \eqref{eq:SumExpSamp} as follows:
  \begin{align*}
    \sum_{j\in \mathcal{J}_{R^{\ast}}}\sum_{\lambda\in\Lambda_{R^{\ast},j}}\exp(-\tfrac{\pi}{2}(|\lambda|-R)^2) &\leq \sum_{j\in \mathcal{J}_{R^{\ast}}}\sum_{\lambda\in\Lambda_{R^{\ast},j}}\iint_{Q(j)}\exp\big(-\tfrac{\pi}{4}\big(|z|-R\big)^2\big)\,dz\\
    &\leq N_{\Lambda}\,e^{\pi}\iint_{\Rt\setminus B(O,R^{\ast}-\sqrt{2})}\exp\big(-\tfrac{\pi}{4}\big(|z|-R\big)^2\big)\,dz\\
    &\leq N_{\Lambda}\,e^{\pi}e^{\frac{\pi R^2}{4}}\iint_{|z|>R^{\ast}-\sqrt{2}}\exp\big(-\tfrac{\pi|z|^2}{8}\big)\,dz\\
    &= 8N_{\Lambda}\,e^{\pi}e^{\frac{\pi R^2}{4}}\exp\big(-\tfrac{\pi(R^{\ast}-\sqrt{2})^2}{8}\big)\\
    &\leq 8N_{\Lambda}\,e^{\frac{5\pi}{4}}\exp\big(-\tfrac{\pi}{4}(\tfrac{(R^{\ast})^2}{4}-R^2)\big).
  \end{align*}
  By taking $C_{\Lambda} = 8N_{\Lambda}\,e^{\frac{5\pi}{4}}$, we get the conclusion of the lemma.  
\end{proof}

\begin{example}\label{ex:locapproxgauss}
  Suppose that the Gabor system $\{\varphi_{0,\lambda}:=\pi(\lambda)\varphi_0\,:\,\lambda\in\Lambda\}$ forms a frame having upper and lower frame bounds $A$ and $B$, respectively, and let the system $\{\widetilde{\varphi_{0,\lambda}}\,:\,\lambda\in\Lambda\}$ be its dual frame. Then, for any $\varepsilon$ between $0$ and $1$, there exists $R_{\varepsilon}>0$ such that if $\Omega$ is a disk centered at the origin with radius $R$, the following inequality holds for all $f\in V_N$:
  \begin{equation}\label{eq:1est}
   \bigg\|f-\sum\limits_{|\lambda|\leq R+R_{\varepsilon}}\langle f,\varphi_{0,\lambda}\rangle \widetilde{\varphi_{0,\lambda}}\bigg\|_2\leq \varepsilon \|f\|_2.
  \end{equation}
  Here, we can take $R_{\varepsilon}\geq -R+\sqrt{4R^2-\frac{16}{\pi}\ln(\varepsilon^2/A^{-1}C_{\Lambda}\sum_{k = 1}^N\frac{1}{\alpha_k^2})}$, where $C_{\Lambda}= 8N_{\Lambda}\,e^{\frac{5\pi}{4}}$.
\end{example}

\begin{proof}
   Following the proof of Proposition \ref{prop:LocGabAppTFLocSubsp}, we have for $\Omega^{\ast}\supset\Omega$,
  \begin{equation*}
   \bigg\|f-\sum_{\lambda\in\Lambda\cap\Omega^{\ast}}\langle f,\varphi_{0,\lambda}\rangle \widetilde{\varphi_{0,\lambda}}\bigg\|_2^2 
      \leq A^{-1}\|f\|_2^2\sum\limits_{k=1}^N\sum\limits_{\lambda\notin\Lambda\cap\Omega^*}\lfrac{1}{\alpha_k^2}|\langle H_{\Omega,\varphi_0}\psi_k,\varphi_{0,\lambda}\rangle|^2.
  \end{equation*}

  We can take $\Omega^{\ast}$ to be a disk centered at the origin with radius $R^{\ast}:=R+R_{\varepsilon}>R$. We use Lemma \ref{lemma:0} (with $\delta = \frac{1}{2}$) and Lemma \ref{lem:EstDecaySumExp}  to estimate the double sum on the right side as follows:
  \begin{align*}
    \sum\limits_{k=1}^N\sum\limits_{\lambda\notin\Lambda\cap\Omega^*}\lfrac{1}{\alpha_k^2}|\langle H\psi_k,\varphi_{0,\lambda}\rangle|^2 
    &\leq \sum\limits_{k=1}^N\sum_{\lambda\notin\Lambda\cap\Omega^{\ast}}\lfrac{1}{\alpha_k^2}\exp(-\tfrac{\pi}{2}\operatorname{dist}(\lambda,\Omega)^2)\|\psi_k\|_2^2\\
    &=\sum\limits_{k=1}^N\sum_{|\lambda|>R^{\ast}}\lfrac{1}{\alpha_k^2}\exp(-\tfrac{\pi}{2}(|\lambda|-R)^2)\\
    &= \bigg(\sum\limits_{k=1}^N\lfrac{1}{\alpha_k^2}\bigg)C_{\Lambda}\exp\big(-\tfrac{\pi}{4}\big(\tfrac{(R^{\ast})^2}{4}-R^2\big)\big).
  \end{align*}
  Now, $\exp\big(-\tfrac{\pi}{4}\big(\tfrac{(R^{\ast})^2}{4}-R^2\big)\big)\leq \varepsilon^2/A^{-1}C_{\Lambda}\sum_{k = 1}^N\frac{1}{\alpha_k^2}$ whenever $R^{\ast}$ is greater than or equal to $\sqrt{4R^2-\frac{16}{\pi}\ln\Big(\varepsilon^2/A^{-1}C_{\Lambda}\sum_{k = 1}^N\frac{1}{\alpha_k^2}}\Big)$, and the conclusion follows.
\end{proof}

The following proposition gives a localization result by means of a given Gabor frame whose atoms are known to have sufficient TF-localization, provided by the condition  $|\mathcal{V}_{\varphi}g(z)|\leq C (1+|z|^{2s})^{-1}$.
\begin{proposition}\label{Pro:TFlocByFrames}
For all $N >0 $ and all $\varepsilon>0$, there exists a set $\Omega^\ast \supset\Omega$ in $\mathbb{R}^2$, such that for all $f\in L^2(\mathbb{R})$ with corresponding orthogonal projection $f_N$ onto the TF-localization subspace $V_N$, the following estimate holds:
\begin{equation}
\left\|f-\sum\limits_{\lambda\in\Lambda\cap\Omega^*}\langle f,g_{\lambda}\rangle\tilde{g_{\lambda}}\right\|_2\leq
\Big(1+\sqrt{\tfrac{B}{A}}\,\Big)\|f-f_N\|_2 + \varepsilon\|f\|_2.
\end{equation}
 
\end{proposition}
\begin{proof}
Since
\begin{align*}
\left\|f-\sum\limits_{\lambda\in\Lambda\cap\Omega^*}\langle f,g_{\lambda}\rangle\tilde{g_{\lambda}}\right\|_2 &\leq
\left\|f-f_N\right\|_2+\left\|f_N -\sum\limits_{\lambda\in\Lambda\cap\Omega^*}\langle f_N,g_{\lambda}\rangle\tilde{g_{\lambda}}\right\|_2\\
  &\phantom{\text{$\leq$ }}+ \left\|\sum\limits_{\lambda\in\Lambda\cap\Omega^*}\langle f_N -f,g_{\lambda}\rangle\tilde{g_{\lambda}}\right\|_2
\end{align*}
the result follows from Proposition \ref{prop:LocGabAppTFLocSubsp} and the boundedness of the associated analysis and synthesis operators.
\end{proof}
As a corollary, we obtain the following result for local approximation of  functions with known time-frequency concentration in a given set $\Omega$ by Gabor frame elements.
\begin{corollary}\label{cor:}
 For fixed $c>1$, let $\psi_k,\,k = 1,\ldots,N$, be the eigenfunctions of $H$ corresponding to eigenvalues $\alpha_k>\frac{c-1}{c}$. 
For  $N$ and $\tilde{\varepsilon}>0$, choose a set $\Omega^\ast \supset\Omega$ as in Proposition \ref{Pro:TFlocByFrames}. Then the following approximation holds for all  functions $f$ which are $(\varepsilon, \varphi)$-concentrated on $\Omega$: 
 \begin{equation}
\left\|f-\sum\limits_{\lambda\in\Lambda\cap\Omega^*}\langle f,g_{\lambda}\rangle\tilde{g_{\lambda}}\right\|_2\leq
\Big(1+\sqrt{\frac{B}{A}}\,\Big)\cdot(\sqrt{c\varepsilon}+\tilde{\varepsilon})\|f\|_2.
\end{equation} 
\end{corollary}

\noindent\textit{Proof}: 
The result follows immediately from Proposition \ref{Prop:eigFuncApp} and Proposition \ref{Pro:TFlocByFrames} .\hfill\qedsymbol\medbreak

\subsection{Local and global frames with TF-localization}
The next results deal with the construction of frames for the subspace $V_N$ of eigenfunctions of $H$, and the whole of $L^2(\R)$, respectively. Denote by $\mathcal{P}_{N}$ the orthogonal projection operator onto the subspace $V_N$.

\begin{proposition}\label{prop:LocFrameIneq}
  If $\varepsilon<1$ and inequality \eqref{eq:2est} is satisfied, then  for all $f\in V_N$,
    \begin{equation}\label{eq:cor1}
     A(1-\varepsilon)^2\|f\|_2^2\leq \sum\limits_{\lambda\in\Lambda\cap \Omega^*} |\langle f,g_{\lambda}\rangle|^2 \leq B\|f\|_2^2,
    \end{equation}
     where $A$ and $B$ are lower and upper frame bounds, respectively, for $\mathcal{G}(g,\Lambda)$. This implies that the system $\{\mathcal{P}_{N}g_{\lambda}\}_{\lambda\in\Lambda\cap \Omega^*}$ forms a frame for $V_N$. More generally, the system $\{\pi(\nu)\mathcal{P}_{N}\pi(\lambda)g\}_{\lambda\in\Lambda\cap \Omega^*}$, where $\nu\in\mathbb{R}^2$, forms a frame for the subspace $\pi(\nu)V_{N}:=\{\pi(\nu)f\,:\,f\in V_N\}$.
\end{proposition}

\begin{proof}
  From Proposition \ref{prop:LocGabAppTFLocSubsp}, we get
  \begin{equation*}
   \|f\|_2-\left\|\sum\limits_{\lambda\in\Lambda\cap\Omega^*}\langle f,g_{\lambda}\rangle\tilde{g_{\lambda}}\right\|_2 \leq \left\|f-\sum\limits_{\lambda\in\Lambda\cap\Omega^*}\langle f,g_{\lambda}\rangle\tilde{g_{\lambda}}\right\|_2\leq \varepsilon\|f\|_2.
  \end{equation*}
  And we obtain
  \begin{align*}
    (1-\varepsilon)^2\|f\|_2^2 &\leq \left\|\sum\limits_{\lambda\in\Lambda\cap\Omega^*}\langle f,g_{\lambda}\rangle\tilde{g_{\lambda}}\right\|_2^2\\
    &\leq \frac{1}{A} \sum\limits_{\lambda\in\Lambda\cap \Omega^*} |\langle f,g_{\lambda}\rangle|^2 \\ 
    &\leq \frac{B}{A}\|f\|_2^2.
  \end{align*}
  For the subspace $\pi(\nu)V_{N}$, we first note that 
  \begin{equation*}
    \|f\|_2 = \|\pi(\nu)f\|_2 \,\,\text{ and }\,\, \langle f,\mathcal{P}_{N}g_{\lambda}\rangle = \langle \pi(\nu)f,\pi(\nu)\mathcal{P}_{N}g_{\lambda}\rangle.
  \end{equation*}
 The inequality in \eqref{eq:cor1} can then be reformulated as
  \begin{equation*}
     A(1-\varepsilon)^2\|\pi(\nu)f\|_2^2\leq \sum\limits_{\lambda\in\Lambda\cap \Omega^*} |\langle \pi(\nu)f,\pi(\nu)\mathcal{P}_{N}g_{\lambda}\rangle|^2 \leq B\|\pi(\nu)f\|_2^2,
  \end{equation*}
  for all $f\in V_N$, or $\pi(\nu)f\in \pi(\nu)V_{N}$.
\end{proof}

  It follows from Proposition \ref{prop:LocFrameIneq} that any function $f\in V_N$ can be completely reconstructed from the samples $\{\langle f,g_{\lambda} \rangle\}_{\lambda\in\Lambda\cap\Omega^*}$. In \cite[Theorem 3.6.16]{fezi98}, a reconstruction procedure was presented where a function on a closed subspace can be reconstructed from restricted Gabor coefficients. Following its approach, we apply the following iterative reconstruction for functions in $V_N$ from the local STFT samples, where $U_{\text{loc}}f = \{\langle f,g_{\lambda}\rangle\}_{\lambda\in\Lambda\cap\Omega^{\ast}}$ and $\tilde{U}_{\text{loc}}^{\ast}c=\sum_{\lambda\in\Lambda\cap\Omega^{\ast}}c_{\lambda}\tilde{g_{\lambda}},\,c=\{c_{\lambda}\}_{\lambda\in\Lambda\cap\Omega^{\ast}}:$ let $f_0 = 0$ and define recursively
  \begin{equation}\label{eq:RecAlg}
      f_n = f_{n-1}+\mathcal{P}_{N}\tilde{U}_{\text{loc}}^{\ast}(\langle f,g_{\lambda}\rangle-U_{\text{loc}}f_{n-1}).
  \end{equation}  
  We implement the above reconstruction procedure in Section \ref{sec:NumEx}, observing the dependence of the performance of the algorithm on the choice of $\Omega^{\ast}$. 

\medskip

\begin{remark}\label{RemHN}\noindent\phantom{-}\\ \vspace{-8pt}
\begin{enumerate}
\item[(a)] Since $\{\mathcal{P}_{N}g_{\lambda}\,:\,\lambda\in\Lambda\cap \Omega^*\}$ is a frame for $V_N$, in the language of \cite{alcamo04}, the system $\{g_{\lambda}\,:\,\lambda\in\Lambda\cap \Omega^*\}$ is also called an \emph{outer frame} for $V_N$. Related terminologies, e.g. \emph{atomic system}, resp. \emph{pseudoframe} for the subspace $V_N$,  appear in \cite{fewe01, liog04}. In particular, by Proposition \ref{prop:LocFrameIneq} and \cite[Theorems $2$ and $3$]{liog04}, the sequence $\{\widetilde{g_{\lambda,V_N}}\}$, where $\widetilde{g_{\lambda,V_N}} = \mathcal{P}_{N}(U\mathcal{P}_{N})^{\dagger}g_{\lambda}$ is called a \emph{dual pseudoframe sequence} for $V_N$ with respect to $\{g_{\lambda}\}.$ This sequence coincides with a dual frame to the frame $\{\mathcal{P}_{N}g_{\lambda}\}$ for $V_N$.    
\item[(b)] It may sometimes be more natural to use, instead of the projection $\mathcal{P}_{N}$, the  \emph{approximate projection} operator $H_N$ defined as: $H_N f := \sum_{k=1}^N\alpha_k \langle f , \psi_k\rangle \psi_k$. Obviously, since we use a finite sequence of positive weights $\sqrt{\alpha_k}$, we obtain an equivalent frame for the subspace $V_N$, if $\mathcal{P}_{N}$ is replaced by $H_N$ in 
Proposition~\ref{prop:LocFrameIneq}.
\end{enumerate}
\end{remark}

We now consider a family of time-frequency localization operators $H^{\mu},\,\mu\in\tilde{\Lambda}$ over the region $\Omega_{\mu}$ with a common window function $\varphi$. In \cite[Theorem 5.10]{doro14}, D\"{o}rfler and Romero showed that under certain conditions on $H^{\mu}$, one can choose $N_{\mu}$ such that 
\begin{equation}\label{eq:Thmdoro14}
  \|f\|_2^2\asymp \sum\limits_{\mu\in\tilde{\Lambda}}\sum\limits_{k = 1}^{N_{\mu}}|\langle f,\psi^{\mu}_k\rangle|^2,\,\,f\in L^2(\R),
\end{equation}
where the functions $\{\psi^{\mu}_k\}$ are eigenfunctions of $H^{\mu}$. We use this result to obtain a frame for $L^2(\R)$ consisting of local frame elements on the time-frequency localized subspaces.

\begin{theorem}\label{thm:locpatchprojframe}

  Let $\{\Omega_{\mu}\}_{\mu\in\tilde{\Lambda}}$ be a family of compact regions in $\Rt$ such that \linebreak $2<\inf_{\mu\in\tilde{\Lambda}}|\Omega_{\mu}|\leq \sup_{\mu\in\tilde{\Lambda}}|\Omega_{\mu}|<\infty$ and $\sum_{\mu\in\tilde{\Lambda}}\chi_{\Omega_{\mu}}\asymp 1$, and let $\varphi\in\So(\R)$ such that $\|\varphi\|_2 = 1$. 

Corresponding to each $\mu\in\tilde{\Lambda}$, choose lattices $\Lambda_{\mu}$ and windows $g^{\mu}\in L^2(\R)$ with  $\|g^{\mu}\|_2 = 1$, $|\cV_{\varphi}g^{\mu}(z)|\leq C_{\mu}(1+|z|^{2s_{\mu}})^{-1}$ for some $C_{\mu}>0$ and $s_{\mu}>1$ such that for all $\mu$ the system  $\mathcal{G}(g^{\mu},\Lambda_{\mu})$ is  a frame for $L^2(\R)$ with frame bounds $A_{\mu}$ and $B_{\mu}$. 

 Denote by $V_{N_{\mu}}$ the span of the first $N_{\mu}$ eigenfunctions $\{\psi^{\mu}_k\}_{k = 1}^{N_{\mu}}$ of $H^{\mu}$ corresponding to the $N_{\mu}$ largest eigenvalues, where each $N_{\mu}$ is chosen so that \eqref{eq:Thmdoro14} holds.  

If $0<\varepsilon_{\mu}<1$ such that $0<\inf_{\mu\in\tilde{\Lambda}} A_{\mu}(1-\varepsilon_{\mu})^2\leq \sup_{\mu\in\tilde{\Lambda}}B_{\mu}<\infty$, then there exist  $\Omega_{\mu}^{\ast}\supset\Omega_{\mu}$ such that the global system $\bigcup_{\mu\in\tilde{\Lambda}}\{\mathcal{P}_{N_{\mu}}\pi(\lambda)g^{\mu}\}_{\lambda\in\Lambda_{\mu}\cap\Omega_{\mu}^{\ast}}$ is a frame for $L^2(\R)$.
\end{theorem}

\begin{proof}
  Let $f\in L^2(\R)$ and $\mu\in\tilde{\Lambda}$. We first note that the conditions on the regions $\Omega_{\mu}$ and $\varphi$ ensure that \eqref{eq:Thmdoro14} holds, cf.~\cite{doro14}. Since $\langle f,\mathcal{P}_{N_{\mu}}\pi(\lambda)g^{\mu}\rangle = \langle \mathcal{P}_{N_{\mu}}f,\pi(\lambda)g^{\mu}\rangle$, it follows from Proposition \ref{prop:LocFrameIneq} that for every $\varepsilon_{\mu}\in(0,1)$, there exists $\Omega_{\mu}^{\ast}\supset\Omega_{\mu}$ such that
  \begin{equation*}
    A_{\mu}(1-\varepsilon_{\mu})^2\sum_{k=1}^{N_{\mu}}|\langle f,\psi^{\mu}_k\rangle|^2\leq \sum\limits_{\lambda\in\Lambda_{\mu}\cap \Omega_{\mu}^*} |\langle f,\mathcal{P}_{N_{\mu}}\pi(\lambda)g^{\mu}\rangle|^2 \leq B_{\mu}\sum_{k=1}^{N_{\mu}}|\langle f,\psi^{\mu}_k\rangle|^2.
  \end{equation*}
  
  By the assumption that $0<\tilde{A}:=\inf_{\mu\in\tilde{\Lambda}} A_{\mu}(1-\varepsilon_{\mu})^2\leq \tilde{B}:=\sup_{\mu\in\tilde{\Lambda}}B_{\mu}<\infty$ and the equivalence in \eqref{eq:Thmdoro14}, we get 
  \begin{equation*}
    \tilde{A}\sum_{\mu\in\tilde{\Lambda}}\sum_{k=1}^{N_{\mu}}|\langle f,\psi^{\mu}_k\rangle|^2\leq \sum_{\mu\in\tilde{\Lambda}}\sum\limits_{\lambda\in\Lambda_{\mu}\cap \Omega_{\mu}^*} |\langle f,\mathcal{P}_{N_{\mu}}\pi(\lambda)g^{\mu}\rangle|^2 \leq \tilde{B}\sum_{\mu\in\tilde{\Lambda}}\sum_{k=1}^{N_{\mu}}|\langle f,\psi^{\mu}_k\rangle|^2,
  \end{equation*}
  and finally $\sum\limits_{\mu\in\tilde{\Lambda}}\sum\limits_{\lambda\in\Lambda_{\mu}\cap \Omega_{\mu}^*} |\langle f,\mathcal{P}_{N_{\mu}}\pi(\lambda)g^{\mu}\rangle|^2\asymp \|f\|_2^2$.
\end{proof}

\begin{remark}\label{rem:GlobalFr}\noindent\phantom{-}\\ \vspace{-8pt}
  \begin{enumerate}
    \item[(a)] For the special case where each $\Omega_{\mu}$ is just the translated region $\mu+\Omega$, if $\{\psi_k\}_{k\in\mathbb{N}}$ is an orthonormal system of eigenfunctions of $H$, then one can choose $N\in\mathbb{N}$ and $\Omega^{\ast}\supset\Omega$ such that $\bigcup_{\lambda\in\Lambda\cap\Omega^{\ast}}\mathcal{G}(\mathcal{P}_{N}g_{\lambda},\tilde{\Lambda})$ is a frame, or a \emph{multi-window} Gabor frame, for $L^2(\R)$.
    \item[(b)] This global system forming a frame obtained from local systems is comparable to \emph{quilted Gabor frames} introduced by D\"{o}rfler in \cite{do11}, the difference being the the projection of the time-frequency dictionary elements onto the time-frequency localized subspaces. In \cite{ro11}, Romero proved results concerning frames for general spline-type spaces from portions of given frames which provide existence conditions for quilted Gabor frames.
    \item[(c)] If $\sum_{\mu\in\tilde{\Lambda}}\chi_{\Omega_{\mu}}\equiv 1$, each $\Lambda_{\mu}$ is a separable lattice, i.e.~$\Lambda_{\mu}=a_{\mu}\mathbb{Z}\times b_{\mu}\mathbb{Z},\,a_{\mu},b_{\mu}>0,$ and the samples $\{\langle f,\pi(\lambda)g^{\mu}\rangle\}_{\lambda\in\Lambda_{\mu}\cap\Omega_{\mu}^{\ast}}$ are given, then $f$ can be recovered completely from the samples if the set of sampling functions $\mathscr{G} = \cup_{\mu\in\tilde{\Lambda}}\mathcal{G}(g^{\mu},\Lambda_{\mu}\cap\Omega_{\mu})$ form a quilted Gabor frame. In \cite{dove14}, the authors presented an approximate reconstruction of $f$ from the given samples using  approximate projection $H^{\mu}_{N_{\mu}}$ in Remark \ref{RemHN}(b). In particular, the following error 
    \begin{equation}\label{eq:ErrApproxRec}
     \bigg\|f- \sum_{\mu\in\tilde{\Lambda}}\sum_{\lambda\in\Lambda_{\mu}\cap\Omega_{\mu}^{\ast}}\langle f,\pi(\lambda)g^{\mu}\rangle H^{\mu}_{N_{\mu}}\pi(\lambda)g^{\mu}\bigg\|_2
    \end{equation}
    was estimated, and numerical experiments were performed in comparison with the method presented in \cite{limarororo13} that used truncated Gabor expansions with weighted coefficients. Through the numerical experiments, it was illustrated how \eqref{eq:ErrApproxRec} can be decreased by having a larger region $\Omega_{\mu}^{\ast}$ or a larger number $N_{\mu}$ of eigenfunctions.
  \end{enumerate}
  
\end{remark}

\section{Numerical Examples}\label{sec:NumEx}
  In this section, we consider examples in the finite discrete case ($\mathbb{C}^L,\,L = 480$) that illustrate the results in the previous sections. The experiments were done in MATLAB using the NuHAG Matlab toolbox available in the following website: \newline\url{http://www.univie.ac.at/nuhag-php/mmodule/}.
  
  Each point set $\Lambda$ that we will use for a Gabor frame is a separable lattice $a\mathbb{Z}_L\times b\mathbb{Z}_L,$ where $a$ and $b$ are divisors of $L$ and also called \emph{lattice parameters} of $\Lambda$. The \emph{redundancy} of $\Lambda$ is given by $\mfrac{L}{ab}$. For more details on Gabor analysis in the finite discrete setting, the reader is referred to the \cite{st98-1}.

  \subsection{Experiment 1}
  We first examine the approximation of time-frequency localized signals by a local Gabor system, in particular, functions lying in the $N$-dimensional subspace $V_N$ of eigenfunctions of $H$, as shown in Proposition \ref{prop:LocGabAppTFLocSubsp}. In this example, we take $\Omega$ to be a disk centered at the origin with radius $80$ and $\varphi$ to be a normalized Gaussian. 
  
  Figure \ref{fig:stftlocsamppts} shows the STFT of a signal in $V_N$ and the sample points taken over circular regions with varying radii, each containing $\Omega$. In each case, the sampling points are obtained by restricting a lattice with parameters $a = b = 20$ (redundancy $1.2$) over the circular region. The error of the approximation $\|\mathcal{P}_{N}-S^{\text{loc}}\mathcal{P}_{N}\|_{\text{Op}}$, where $S^{\text{loc}}$ is a truncated tight frame operator, is shown in Table \ref{tab:stftlocapproxerr} below.

  \begin{figure}[t!hp]
	  \includegraphics[width=0.6\linewidth]{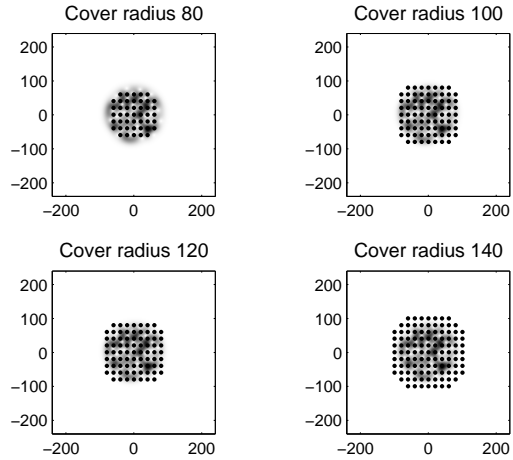}
	  \caption{Sampling points over various enlargements of the covering region.}\label{fig:stftlocsamppts}
  \end{figure}\bigskip

\begin{table}[t!hp]
  \begin{tabular}{|c|c|c|}
    \hline Cover radius & No. of samp. pts. & Op. norm error \\[0.1em] \hline
      $80$  & $45$  & $0.9650$ \\[0.1em]
      $100$ & $77$  & $0.1105$ \\[0.1em]	
      $120$ & $109$ & $0.0194$ \\[0.1em]
      $140$ & $145$ & $0.0031$ \\ \hline
  \end{tabular}\\\bigskip

  \caption{Error $\|\mathcal{P}_{N}-S^{\text{loc}}\mathcal{P}_{N}\|_{\text{Op}}$ over varying radii for the disk $\Omega$}\label{tab:stftlocapproxerr}
  \end{table}

  We saw in Proposition \ref{prop:LocFrameIneq} that if $\varepsilon<1$, corresponding to the operator norm $\|\mathcal{P}_{N}-S^{\text{loc}}\mathcal{P}_{N}\|_{\text{Op}}$ being less than $1$, then the local Gabor system projected into $V_N$ forms a frame for $V_N$ so perfect reconstruction is possible by the reconstruction algorithm \eqref{eq:RecAlg}. The performance of the reconstruction algorithm is shown in Figure \ref{fig:locsampalg}. As expected, the larger the covering region, the faster the convergence.

  \begin{figure}[t!hp]
	  \includegraphics[width=0.6\linewidth]{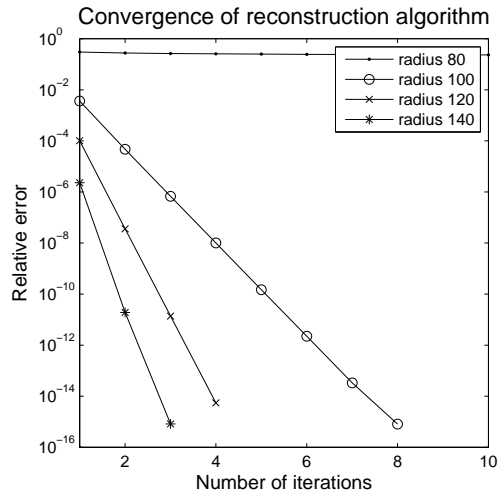}
	  \caption{Convergence of the reconstruction algorithm from the local samples with the same lattice parameters but with varying radii of the covering regions.}\label{fig:locsampalg}
  \end{figure}
  
  \subsection{Experiment 2}\label{Exp2}
  In this next experiment, we look at an example of how the collection of local Gabor systems can form a frame given that the sum of the characteristic functions over the regions is bounded above and below by a positive number. Figure \ref{fig:reg1-10} shows ten regions in the TF-plane and Figure \ref{fig:sumcharfunc} shows its sum.

  \begin{figure}[t!hp]
	  \includegraphics[width=\linewidth]{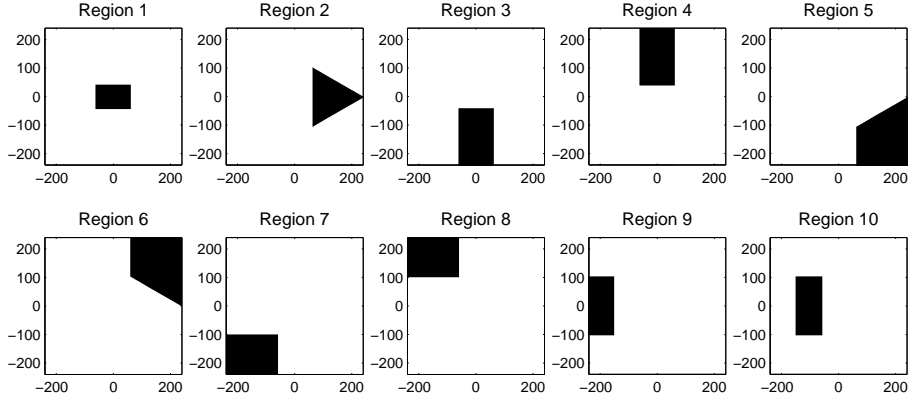}
	  \caption{Ten regions that partition the time-frequency plane.}\label{fig:reg1-10}
  \end{figure}

  \begin{figure}[t!hp]
	  \includegraphics[width=0.4\linewidth]{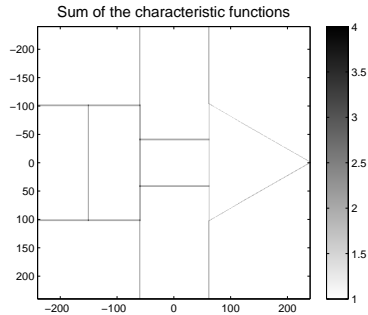}
	  \caption{Sum of the characteristic functions over the ten regions.}\label{fig:sumcharfunc}
  \end{figure}

  Sample points are then taken over sets that contain each region, where different lattices are used for each set. The lattice parameters assigned to each set are summarized in Table \ref{tab:locpatchlatt}, and the sample points are depicted in Figure \ref{fig:locpatchsamp}. The left image shows sample points obtained by restricting each lattice over the regions themselves, while the samples in the right image are obtained from the restriction over larger sets containing each region, thus producing more overlap. Tight windows are used corresponding to each set of restricted lattice points.

  \begin{table}[t!hp]
    \begin{tabular}{|c|c||c|c|}
      \hline Region &  $(a,b)$ & Region & $(a,b)$ \\[0.1em] \hline
      $1$ & $(20,20)$ &  $6$ & $(15,15)$ \\[0.1em]
      $2$ & $(16,20)$ &  $7$ & $(12,15)$ \\[0.1em]
      $3$ & $(20,16)$ &  $8$ & $(12,12)$ \\[0.1em]
      $4$ & $(16,16)$ &  $9$ & $(10,12)$ \\[0.1em]
      $5$ & $(15,16)$ & $10$ & $(10,10)$ \\ \hline
    \end{tabular}\\\bigskip

    \caption{Lattice parameters over the different regions.}\label{tab:locpatchlatt}
  \end{table}

  \begin{figure}[t!hp]
	  \includegraphics[width=0.48\linewidth]{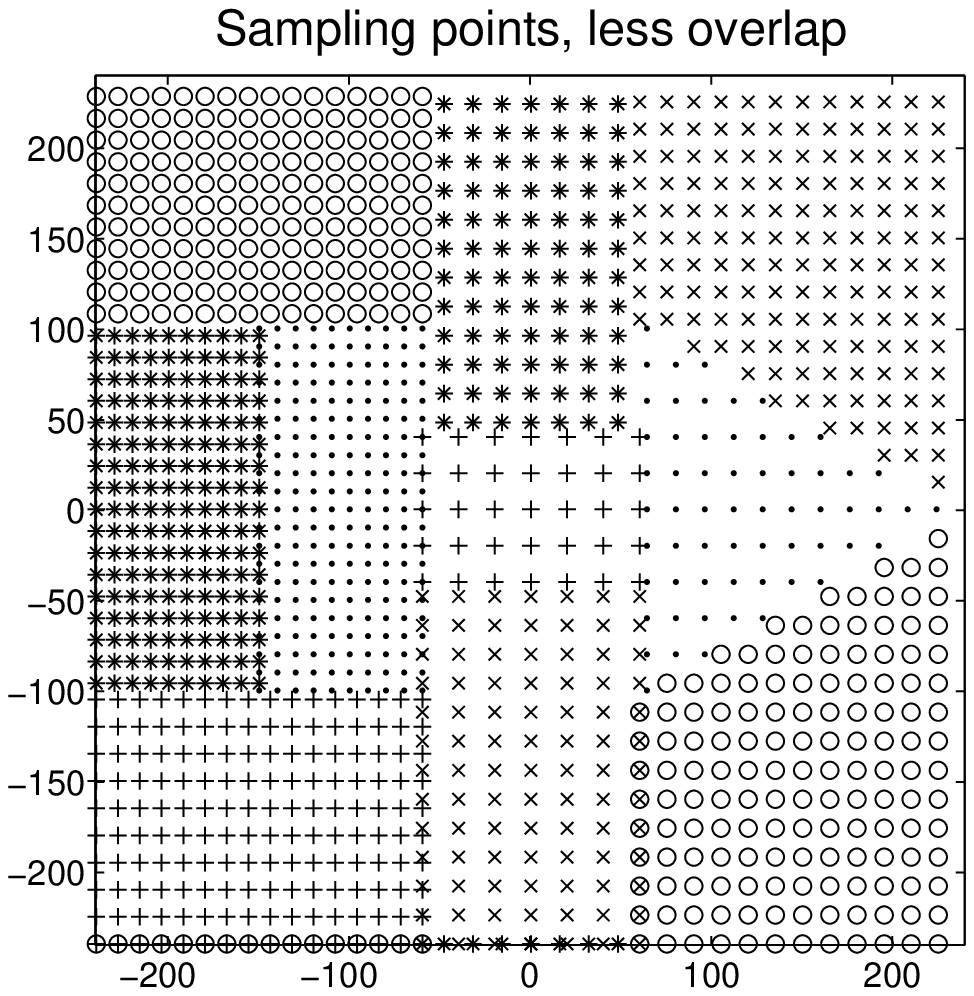}\,\,\includegraphics[width=0.48\linewidth]{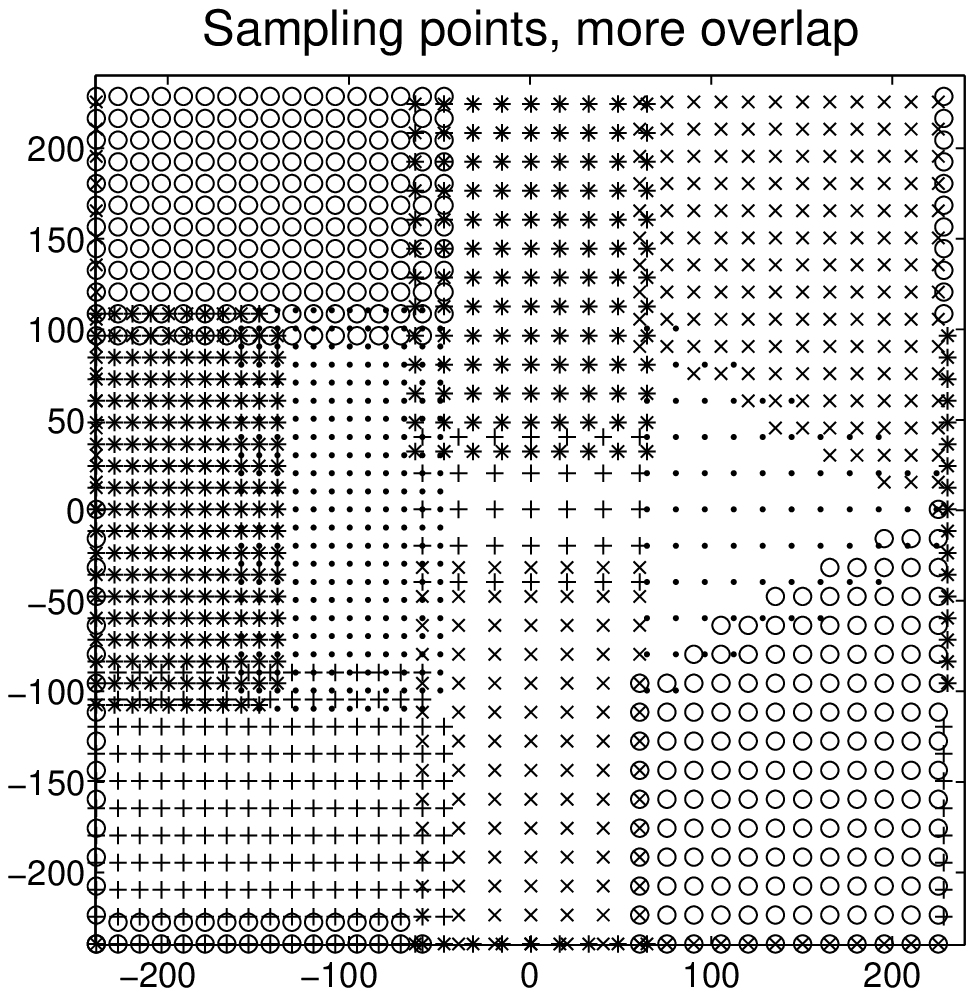}
	  \caption{Sampling points on the different local patches.}\label{fig:locpatchsamp}
  \end{figure}

  We form a quilted Gabor frame from the collection of local Gabor systems. And by projecting each local Gabor system onto the local subspace corresponding to each region, we likewise obtain a global frame as in Theorem \ref{thm:locpatchprojframe}. The average of the relative error $\mfrac{\|f-S_if\|_2}{\|f\|_2}$ when the frame operators $S_1$ and $S_2$, corresponding to the quilted Gabor frame (i.e.~without projection) and the global frame (i.e.~with projection), respectively, are applied to a random signal $f$ are shown in Table \ref{tab:aveerr}.

  \begin{table}[t!hp]
    \begin{tabular}{|c|c|c|}
      \hline  & without projection &  with projection \\[0.1em] \hline	
      Less overlap & $0.2610$ &  $0.1687$ \\[0.1em]
      More overlap & $0.5840$ &  $0.1709$ \\ \hline
    \end{tabular}\\\bigskip

    \caption{Average of the error in applying the frame operator to a random signal (average of $1000$ attempts).}\label{tab:aveerr}
  \end{table}

  In both cases of less and more overlap, projecting onto the TF-localized subspaces decreases the relative error between the signal and the approximation by the frame operator. Note that in both quilted Gabor frame and the global frame with projection, having more overlap increases the relative error since we are just comparing $f$ with $S_if$. Since we are dealing with frames, perfect reconstruction (up to numerical error) is possible via the frame algorithm cf.~\cite[Algorithm 5.1.1]{gr01}. 

  We first compare the respective condition numbers of the frame operators for the cases of less and more overlap. The values are shown in Table \ref{tab:condnumlocpatchfr}. Once again, in both quilted Gabor frame and the global frame with projection, having more overlap improves the condition number. Note that the large condition number for the frame operator corresponding to the global frame with less overlap can be attributed to the lower frame bound in Theorem \ref{thm:locpatchprojframe}, which is related to the set $\Omega_{\mu}^{\ast}$ that covers the region $\Omega_{\mu}$ - a smaller region $\Omega_{\mu}^{\ast}$ implies a smaller lower frame bound.

\begin{table}[t!hp]
  \begin{tabular}{|c|c|c|}
    \hline  & without projection &  with projection \\[0.1em] \hline
    Less overlap & $5.1429$ &  $16.0406$ \\[0.1em]
    More overlap & $3.5472$ &  $1.9845$ \\ \hline
  \end{tabular}\\\bigskip

  \caption{Condition numbers of the resulting frame operators.}\label{tab:condnumlocpatchfr}
\end{table}


  Figure \ref{fig:convfralglocpatch} compares the convergence of the frame algorithm for the four cases considered. 

\begin{figure}[t!hp]
  \includegraphics[width=220pt]{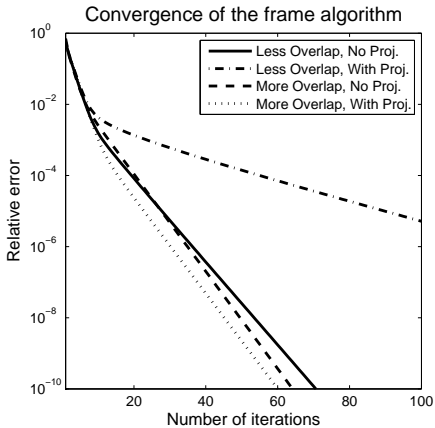}
  \caption{Convergence of the frame algorithm.}\label{fig:convfralglocpatch}
\end{figure}

\subsection{Experiment 3}\label{Exp3}
In this final experiment, we illustrate the approximate reconstruction of a signal from analysis coefficients obtained from a union of tight Gabor systems, each restricted on an enlarged region covering a given region of interest, forming a quilted Gabor frame (see Remark \ref{rem:GlobalFr}(c)). 

Similar to the experiment in \cite{dove14} we consider four rectangular regions and associate tight Gabor frames to each one:
\begin{enumerate}
  \item[1.] $\mathcal{G}(g_1,20,8)$ on the region corresponding to lower frequency and time $t\leq L/2$;
  \item[2.] $\mathcal{G}(g_2,24,10)$ on the region corresponding to lower frequency and time $t> L/2$;
  \item[3.] $\mathcal{G}(g_3,12,24)$ on the region corresponding to higher frequency and time $t\leq L/2$;
  \item[4.] $\mathcal{G}(g_4,15,20)$ on the region corresponding to higher frequency and time $t> L/2$.
\end{enumerate}
The sample points on the four regions are depicted in Figure \ref{fig:locpatchsampExp3}.

\begin{figure}[t!hp]
  \includegraphics[width=220pt]{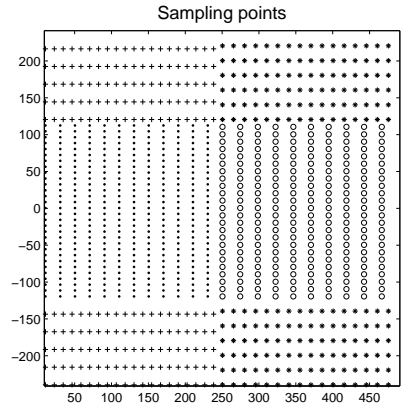}
  \caption{Sampling points over the four regions.}\label{fig:locpatchsampExp3}
\end{figure}

We apply an approximate projection onto the subspaces of eigenfunctions of the time-frequency localization operators on the regions and compute the relative error from the approximate reconstruction. We compare the relative errors over varying overlap $b$ (the amount of increase in the length of a side of the rectangular region) in Figure \ref{fig:ApproxErrorVsAmtOverlap} for three different eigenspace dimensions (nEV). We also include the relative error obtained from re-synthesizing with the same quilted Gabor frame elements.

\begin{figure}[t!hp]
  \includegraphics[width=220pt]{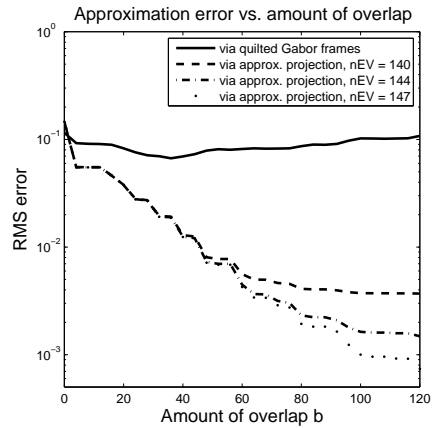}
  \caption{Approximation error vs. amount of overlap.}\label{fig:ApproxErrorVsAmtOverlap}
\end{figure}

For the case of quilted Gabor frames, an overlap would initially decrease the relative error in the approximate reconstruction but since we are just essentially getting the error from applying the frame operator on the signal, more overlap in the regions would eventually lead to an increase in the relative error. For the cases with approximate projection, we see the decrease in the relative error as the overlap amount increases. Moreover, the relative error improves as the eigenspace dimension is increased.

\section{Conclusions and Perspectives}\label{sec:Conc}
In this paper we investigated the representation of time-frequency localized functions by means of sampling in the time-frequency domain. Motivated by the problem of providing dictionaries with good concentration within prescribed regions in time-frequency, we constructed frames of localized time-frequency atoms. The  improved localization is obtained by means of projections onto eigenspaces corresponding to time-frequency localization operators. Numerical experiments illustrated the promising potential of the proposed method: providing good reconstruction/approximation quality while preserving the good localization property. However, for applications to real signals, the proposed method would entail having to compute for eigenvalues and eigenfunctions of large matrices, which may be numerically cumbersome. The study of efficient numerical methods for the evaluation/approximation of eigenvalues and eigenfunctions of time-frequency localization operators would be a topic of future work.

\bibliographystyle{abbrv}

\end{document}